\documentclass{amsart}

\usepackage[T1]{fontenc}
\usepackage{microtype}
\usepackage{lmodern}
\usepackage[colorlinks=true,urlcolor=blue, citecolor=red,linkcolor=blue,linktocpage,pdfpagelabels, bookmarksnumbered,bookmarksopen]{hyperref}
\usepackage[hyperpageref]{backref}
\usepackage{amsthm} %for citing inside theorem header
\usepackage{latexsym,amsmath,amssymb}

\usepackage{accents}
\usepackage{esint}

\usepackage{soul}
\usepackage{mathtools} %For \chi with a much lower index (\mychi)
\usepackage{xparse} %For a new definiton of \chi with a slightly lower index
\usepackage[capitalize]{cleveref}

\usepackage[shortlabels]{enumitem}

\usepackage[textsize=small]{todonotes}
\title[$n$-Laplace systems in divergence form]{A note on limiting Calderon-Zygmund theory for transformed $n$-Laplace systems in divergence form}
\author{Dorian Martino}
\address[Dorian Martino]{Institut de Mathématiques de Jussieu, Université Paris Cité, Bâtiment
	Sophie Germain, 75205 Paris Cedex 13, France}
\email{dorian.martino@imj-prg.fr}

\author{Armin Schikorra}
\address[Armin Schikorra]{Department of Mathematics,
	University of Pittsburgh,
	301 Thackeray Hall,
	Pittsburgh, PA 15260, USA}
\email{armin@pitt.edu}

\newcommand{\N}{{\mathbb N}}

\newcommand\Cr{{\mathcal C}}

\newtheorem{theorem}{Theorem}
\newtheorem{lemma}[theorem]{Lemma}
\newtheorem{corollary}[theorem]{Corollary}

\theoremstyle{definition}

\theoremstyle{remark}
\newtheorem{remark}[theorem]{Remark}

\newcommand{\g}{\nabla}
\newcommand{\dr}{\partial}
\newcommand{\di}{\mathrm{div }}
\newcommand{\R}{\mathbb{R}}

\newcommand{\HI}{\mathcal{H}}

\newcommand{\barint}{
	\rule[.036in]{.12in}{.009in}\kern-.16in \displaystyle\int }

\newcommand{\barcal}{\text{$ \rule[.036in]{.11in}{.007in}\kern-.128in\int $}}

\def\mvint_#1{\mathchoice
	{\mathop{\vrule width 6pt height 3 pt depth -2.5pt
			\kern -8pt \intop}\nolimits_{\kern -3pt #1}}%
{\mathop{\vrule width 5pt height 3 pt depth -2.6pt
		\kern -6pt \intop}\nolimits_{#1}}%
{\mathop{\vrule width 5pt height 3 pt depth -2.6pt
		\kern -6pt \intop}\nolimits_{#1}}%
{\mathop{\vrule width 5pt height 3 pt depth -2.6pt
		\kern -6pt \intop}\nolimits_{#1}}}

		\numberwithin{theorem}{section} \numberwithin{equation}{section}

		\renewcommand{\subjclassname}{%
\textup{2020} Mathematics Subject Classification}

\newcommand{\lap}{\Delta }

\usepackage{scalerel}[2014/03/10]
\usepackage[usestackEOL]{stackengine}
\def\avint{\,\ThisStyle{\ensurestackMath{%
		\stackinset{c}{.2\LMpt}{c}{.5\LMpt}{\SavedStyle-}{\SavedStyle\phantom{\int}}}%
	\setbox0=\hbox{$\SavedStyle\int\,$}\kern-\wd0}\int}
	\renewcommand{\div}{\operatorname{div}}
	
	\let\latexchi\chi
	\makeatletter
	\renewcommand\chi{\@ifnextchar_\sub@chi\latexchi}
	\newcommand{\sub@chi}[2]{% #1 is _, #2 is the subscript
\@ifnextchar^{\subsup@chi{#2}}{\latexchi^{}_{#2}}%
}
\newcommand{\subsup@chi}[3]{% #1 is the subscript, #2 is ^, #3 is the superscript
\latexchi_{#1}^{#3}%
}
\makeatother
\newcommand{\eps}{\varepsilon}
\newcommand{\ve}{\varepsilon}

\begin{document}

\begin{abstract}
	We consider rotated $n$-Laplace systems on the unit ball $B_1 \subset \mathbb{R}^n$ of the form
	\begin{align*}
		-\mathrm{div}\left( Q|\nabla u|^{n-2} \nabla u\right) = \mathrm{div}(G),
	\end{align*}
	where $u\in W^{1,n}(B_1;\mathbb{R}^N)$, $Q\in W^{1,n}(B_1;SO(N))$ and $G\in L^{\left( \frac{n}{n-1},q \right)}(B_1;\mathbb{R}^n\otimes \mathbb{R}^N)$ for some $0<q<\frac{n}{n-1}$. We prove that $\nabla u\in L^{(n,q(n-1))}_{loc}$ with estimates. As a corollary, we obtain that solutions to $\Delta_n u \in \mathcal{H}^1$, where $\mathcal{H}^1$ is the Hardy space, have a higher integrability, namely $\nabla u \in L^{(n,n-1)}_{loc}$.
\end{abstract}

\maketitle
\tableofcontents

{\bf \subjclassname:} 35B65, 35J70, 35B45, 35H30, 35J70, 35J92, 46E35.

\section{Introduction}
In the theory of critical harmonic maps into manifolds $u: B^2 \subset \R^2 \to \mathcal{M}$, cf. \cite{H91,H02,Riv07}, the Hardy space is an important tool, since any map $u \in W^{1,2}(B^2,\R^N)$ satisfying
\[
\lap u \in \mathcal{H}^1
\]
is continuous, and the Hardy space naturally appears via commutators and div-curl terms \cite{CLMS}.
The continuity statement is false for the $n$-Laplacian when $n \geq 3$: in 1995, Firoozye \cite{F95} exhibited discontinuous maps $u \in W^{1,n}(B^n,\R^N)$ satisfying
\[
\div(|\nabla u|^{n-2} \nabla u) \in \mathcal{H}^1.
\]
Indeed the regularity theory for $n$-harmonic maps into general manifolds is an interesting and difficult open question, see \cite{SS17} for an overview and \cite{MPS23,MS23} for two recent results.

Formally the result of \cite{F95} is not surprising: ``Inverting'' the $\div$ (i.e. pretending it to be the half-Laplacian) we can wishfully hope for
\[
|\nabla u|^{n-1} \in L^{\left( \frac{n}{n-1},1 \right)},
\]
where $L^{\left( \frac{n}{n-1},1 \right)}$ is a Lorentz space -- since Sobolev embedding implies that $(-\lap)^{-\frac{1}{2}} \mathcal{H}^1 \subset L^{\left( \frac{n}{n-1},1 \right)}$ -- and thus
\[
\nabla u\in L^{(n,n-1)}.
\]
While $\nabla u \in L^{(n,1)}$ implies $u$ is continuous (this is what we have in the case $n=2$), for $n \geq 3$ standard function space theory tells us that there are many counterexamples $u\in W^{1,n}$ satisfying $\nabla u \in L^{(n,n-1)}$ but $u \not \in C^0$.

The purpose of this short note is to make this intuition more precise, somewhat giving a positive version of Firoozye's example. It also extends known limiting results for the $p$-Laplacian in \cite{ACS15,ACS17} (observe, however we restrict to $p=n$). Our main result is the following statement, where $B_r = B(0,r)$ is the ball of $\R^n$ with radius $r$ centered at the origin.
\begin{theorem}\label{th:main}
	Let $q\in(0,\frac{n}{n-1})$. There exists a small $\eps=\eps(n,N,q)$ and even smaller $\gamma=\gamma(\eps)$ we have the following. \\
	
	Assume $G\in L^{\left( \frac{n}{n-1}, q\right)}(B_1;\R^n\otimes \R^N)$, $A\in W^{1,n}(B_1;GL(N))$, $\|A\|_{L^\infty} + \|A^{-1}\|_{L^\infty} \leq 2 N^2$, and $u \in W^{1,n}(B_1;\R^N)$ satisfy the system
	\begin{align*}
		-\di(|\g u|^{n-2} A\g u) = \di G\ \text{in }B_1,
	\end{align*}
	with the bound
	\[
	\|\g A^{-1}\|_{L^n(B_1)}+\|\g A\|_{L^n(B_1))} \leq \gamma.
	\]
	Then, for every $\theta\in(0,\frac{1}{4})$, it holds $\g u \in L^{(n,q(n-1))}(B_\theta)$ with the following estimate
	\begin{align*}
		\|\g u\|_{L^{(n,q(n-1))}(B_\theta) }^{n-1} \leq c(n,N,q,\theta)\left(\left\| G \right\|_{L^{\left( \frac{n}{n-1}, q  \right)}(B_1) } + \|\g u\|_{L^{n-\eps}(B_1)}^{n-1} \right).
	\end{align*}
\end{theorem}

\begin{remark}\label{rk:main_thm}
	We can make a few remarks on the above statement:
	\begin{enumerate}
		\item The quantity $2N^2$ in the bound $\|A\|_{L^\infty} + \|A^{-1}\|_{L^\infty} \leq 2 N^2$ is arbitrary and can be replaced by any large constant depending only on $n,N$.
		\item We include the term $A \in GL(N)$, because the theory of $n$-harmonic maps allows for a change of gauge, either the Uhlenbeck--Coulomb gauge $A \in SO(N)$ \cite{Uhl82}, or the Rivi\`ere's gauge $A \in GL(N)$ \cite{Riv07}, for the relation see also \cite{Sch10}.
		\item  In the case $q=\frac{1}{n-1}$ of \Cref{th:main}, we deduce that $u$ is continuous.
		\item With the arguments of the proof only provide a constant $c$ in the last estimate which goes to $+\infty$ as $\theta \to \frac{1}{4}$.
	\end{enumerate}
\end{remark}

In terms of the Hardy space we obtain the following as an immediate consequence
\begin{corollary}\label{co:nLap_H1}
	Assume $f\in \HI^1_{loc}(B_1;\R^N)$,  $A\in W^{1,n}(B_1;GL(N))$, $\|A\|_{L^\infty} + \|A^{-1}\|_{L^\infty} \leq 2N^2$, and $u \in W^{1,n}(B_1;\R^N)$ satisfy the system
	\begin{align*}
		-\di(|\g u|^{n-2} A\g u) = f\ \text{in }B_1,
	\end{align*}
	with the bound
	\[
	\|\g A^{-1}\|_{L^n(B_1)}+\|\g A\|_{L^n(B_1)} \leq \gamma.
	\]
	Then, for every $\theta\in(0,\frac{1}{4})$, it holds $\g u \in L^{(n,n-1)}(B_\theta)$ with the following estimate
	\begin{align*}
		\|\g u\|_{L^{(n,n-1)}(B_\theta) }^{n-1} \leq c(n,N,\theta)\left(\left\| f \right\|_{\HI^1(B_1) } + \|\g u\|_{L^{n-\eps}(B_1)}^{n-1} \right).
	\end{align*}
\end{corollary}

\begin{proof}
	We solve $\lap \phi = f$ in $B_1$, with $\phi = 0$ on $\dr B_1$. Since $f\in\HI^1$, it holds $\g \phi \in L^{\left( \frac{n}{n-1}, 1\right)}(B_1)$. \Cref{co:nLap_H1} follows from \Cref{th:main} with $G = \g \phi$.
\end{proof}
Observe that \Cref{co:nLap_H1} is sharp in dimension $n=2$. We will argue that \Cref{co:nLap_H1} is also sharp in dimension $n\geq 3$ in some sense in \Cref{s:optimal}. Observe also that the bound $2N^2$ in the $L^\infty$-estimate of $A$ is arbitrary, see \Cref{rk:main_thm}.\\

{\bf Outline}
The starting point of our arguments is based on recent estimates by the authors \cite{MS23}, which in turn are strongly motivated by Kuusi and Mingione's seminal \cite{KuusiMingione18}, combined with covering arguments to estimate level sets. Then we adapt ideas of \cite{mingione2010} to obtain our result. In \Cref{s:prelim} we recall the definition of Lorentz spaces and the necessary preliminary results. In \Cref{s:levelset} we prove \Cref{th:main}. In \Cref{s:optimal}, we discuss the regularity of Firoozye's example.\\

{\bf Acknowledgment}
D.M.'s research was partially funded by ANR BLADE-JC ANR-18-CE40-002. D.M. thanks Paul Laurain for his constant support and advice. A.S. is an Alexander-von-Humboldt Fellow. A.S. is funded by NSF Career DMS-2044898.

\section{Preliminary estimates}\label{s:prelim}

In this section, we define some notations and recall the necessary preliminary estimates on Lorentz spaces, $p$-harmonic maps and maximal functions.\\

In the rest of the paper, we will denote $B(x,r)\subset \R^n$ the ball of radius $r>0$ and center $x\in\R^n$. If $x=0$, we will denote $B_r = B(0,r)$. If $\lambda>0$ and $B=B(x,r)$ is a ball, we will denote $\lambda B := B(x,\lambda r)$. \\

We now recall the definitions and relevant properties of Lorentz spaces. For further reading, see for instance \cite[Section 1.4]{GrafakosClassical}. Let $\Omega\subset \R^n$ be an open set. Given a function $f : \Omega \to \R$, we define its decreasing rearrangement $f^* : [0,|\Omega|)\to [0,\infty)$ by
\begin{align*}
	\forall t>0, \ \ \ f^*(t) := \inf\left\{ \lambda\geq 0 : |\{ x\in \Omega : |f(x)|>\lambda\}| \leq t \right\}.
\end{align*}
Given $p\in(0,\infty)$ and $q\in(0,+\infty]$, a function $f:\Omega\to \R$ belongs to the Lorentz space $L^{(p,q)}(\Omega)$ if the following quantity is finite :
\begin{align*}
	\| f\|_{L^{(p,q)}(\Omega)} := \left\{ \begin{array}{l l}
		\left(  \int_0^\infty \left( f^*(s) s^\frac{1}{p} \right)^q \frac{ds}{s} \right)^\frac{1}{q} & \text{ if }q<\infty,\\
		\sup_{s>0} s^\frac{1}{p} f^*(s) & \text{ if }q=\infty.
	\end{array}
	\right.
\end{align*}
Lorentz spaces are refinements of the Lebesgue spaces in the following sense. Given $p\geq 1$ and $0< q<r \leq +\infty$ and $|\Omega|<\infty$, it holds $L^{(p,p)}(\Omega) = L^p(\Omega)$ and $L^{(p,q)}(\Omega)\subset L^{(p,r)}(\Omega)$.
Given $r>0$, $p>1$ and $q\in (0,\infty]$, it holds
\begin{align*}
	\|f^r \|_{L^{(p,q)}(\Omega)} &= \|f\|_{L^{(pr,qr)}(\Omega)}^r.
\end{align*}

From the regularity of $p$-harmonic maps (vectorial, but with unconstrained target), cf. \cite[Theorem 3.2]{uhlenbeck1997}, we have
\begin{theorem}\label{th:integrability_pharmonic}
	Let $p\in(1,n]$ and $\theta_0\in(0,1)$. There exists $c=c(n,N,p,\theta_0)>0$ such that the following holds. Consider a ball $B(x,r)\subset \R^n$ and $v\in W^{1,p}(B(x,r);\R^N)$ such that $\lap_p v =0$ on $B(x,r)$. Then it holds
	\begin{align*}
		\|\g v\|_{L^\infty(B(x,\theta_0 r))} \leq c\left( \mvint_{B(x,r)} |\g v|^p \right)^\frac{1}{p}.
	\end{align*}
\end{theorem}

\begin{remark}\label{rk:choice_theta}
	The above estimate has been proved for $\theta_0 = \frac{1}{4}$ in \cite[Theorem 3.2]{uhlenbeck1997}. By a covering argument, we can choose $\theta_0\in(0,1)$ arbitrarily, up to increase the constant $c$.
\end{remark}

The following is the initial estimate we need for our purposes, which was proved in \cite[Corollary 5.2.]{MS23}. 
\begin{lemma}\label{la:initial_estimate_difference}
	Let $\sigma\in(0,1)$.  There exists $\eps_1=\eps_1(n,N,\sigma) > 0$ such that the following holds.
	
	For any $\eps \in (0,\eps_1)$ there exists $\gamma_1 = \gamma_1(n,N,\eps,\sigma) > 0$ with the following properties.
	
	There exists $C_0=C_0(n,N,\sigma,\eps)>0$ such that the following hold.
	
	Assume that $u \in W^{1,n}(B(x,r);\R^N)$ satisfies
	\[
	\div(A|\nabla u|^{n-2} \nabla u) = \di G\quad \text{in $B(x,r)$}.
	\]
	where $A \in W^{1,n}(B(x,r);GL(N))$, $\|A\|_{L^\infty} + \|A^{-1}\|_{L^\infty} \leq 2N^2$ and
	\begin{align}\label{eq:assumption_Q}
		\|\nabla A\|_{L^n(B(x,r))} + 		\|\nabla A^{-1}\|_{L^n(B(x,r))} \leq \gamma_1.
	\end{align}
	
	There exists a radius $\rho\in\left[ \frac{1}{2}r,\frac{3}{4}r \right]$ such that if $v\in W^{1,n}(B(x,\rho);\R^N)$ satisfies
	\begin{align*}
		\left\{ \begin{array}{c l}
			\lap_n v = 0 & \text{in }B(x,\rho),\\
			v = u & \text{on }\dr B(x,\rho),
		\end{array}
		\right.
	\end{align*}
	then it holds
	\begin{align}\label{initial_estimate_difference}
		\left( \mvint_{B(x,\rho)} |\g u - \g v|^{n-\ve} \right)^\frac{1}{n-\ve} &\leq C_0(\sigma,\eps)\left(\mvint_{B(x,r)} |G|^\frac{n-\eps}{n-1}\right)^\frac{1}{n-\eps}
		+ \sigma \left( \mvint_{B(x,r)} |\g u|^{n-\ve} \right)^\frac{1}{n-\ve}.
	\end{align} 
\end{lemma}
\begin{remark}
	In the statement of \cite[Corollary 5.2]{MS23}, the integral quantity involving $G$ in the right-hand side of \eqref{initial_estimate_difference} is $\left(\mvint_{B(x,r)} |G|^\frac{n}{n-1}\right)^\frac{1}{n}$. However, the above estimate is obtained by following step by step the proof of \cite[Lemma 4.2]{MS23} in the case $f=0$. The only change is the estimate of the term $I$ which has to be replaced, with the notations of the proof of \cite[Lemma 4.2]{MS23}, with the Hölder inequality $I\leq C\|G\|_{L^\frac{n-\eps}{n-1}} \|\g a\|_{L^\frac{n-\eps}{1-\eps}}$.
\end{remark}

We will work on dyadic cubes and balls. To that extent, we give some definitions. Given $a=(a_1,\ldots,a_n)\in\R^n$ and $\ell>0$, we consider the cube centred in $a$ and having side-length $\ell$:
\begin{align*}
	Q_{\ell}(a) := \left[a_1-\frac{\ell}{2},a_1+\frac{\ell}{2} \right] \times \cdots \times \left[a_n-\frac{\ell}{2},a_n+\frac{\ell}{2} \right] \subset \R^n.
\end{align*}
Given $r>0$, we will denote $rQ_{\ell}(a) := Q_{r\ell}(a)$. Given a ball $B\subset \R^n$, the \textit{inner cube} of $B$ is the cube $Q\subset B$ having the same centre and maximal side-length. Dyadic subcubes of a cube $Q = Q_{\ell}(a)$ are defined by induction as follows. We denote $\Cr_0 := \{Q\}$ and $\Cr_1$ the family of subcubes of $Q$ obtained by dividing $Q$ in $2^n$ cubes having disjoint interior and such that each of them have side-length $\frac{\ell}{2}$. Given an integer $k\geq 1$, assume that $\Cr_k$ have been defined. We consider the family $\Cr_{k+1}$ of subcubes of $Q$ obtained by dividing each cube $\tilde{Q}\in \Cr_k$ in $2^n$ subcubes having disjoint interior and such that the side-length of each of these subcubes is equal to half of the side-length of $\tilde{Q}$. A subcube $\tilde{Q}\subset Q$ is called \textit{dyadic} if $\tilde{Q}\in \bigcup_{k\geq 1} \Cr_k$. Given $k\geq 1$ and $Q_1\in \Cr_k$, there exists a unique $Q_0\in \Cr_{k-1}$, called the \textit{predecessor} of $Q_1$, such that $Q_1\subset Q_0$. A predecessor is defined only for strict subcubes of $Q$ and is always a dyadic subcube of $Q$.\\

We will work with a covering argument, for this we need the the following result from \cite[Lemma 1.2]{caffarelli1998}.
\begin{lemma}\label{la:size_subset}
	Let $Q_0\subset \R^n$ be a cube. Assume that $X\subset Y\subset Q_0$ are measurable sets such that the following properties hold
	\begin{enumerate}
		\item\label{It1} there exists $\delta>0$ such that $|X|<\delta|Q_0|$,
		\item\label{It2} if $Q\subsetneq Q_0$ is a dyadic subcube, then the inequality $|X\cap Q|>\delta|Q|$ implies that the predecessor $\tilde{Q}$ of $Q$ is contained in $Y$.
	\end{enumerate}
	Then it holds $|X| < \delta|Y|$.
\end{lemma}

We also have the following result for the uncentered, restricted maximal function operator, see \cite[Theorem 7]{mingione2010}. Given $f\in L^1(B_1)$, we set
\begin{align*}
	M_{B_1}f(x) := \sup \left\{ \mvint_{B(y,r)} |f| : x\in B(y,r)\subset B_1  \right\}.
\end{align*}

\begin{lemma}\label{th:boundedness_maximal_function}
	Let $t>1$ and $q\in(0,\infty]$. There exists a constant $c=c(n,t,q)>0$ such that the following holds. \\
	Consider a ball $B_1\subset \R^n$ and $g\in L^{(t,q)}(B_1)$. Then,
	\begin{align*}
		\|M_{B_1}g\|_{L^{(t,q)}(B_1)} \leq c \|g\|_{L^{(t,q)}(B_1)}.
	\end{align*}
\end{lemma}

\section{Level-set estimates: Proof of Theorem~\ref{th:main}}\label{s:levelset}

In this section we adapt the techniques of \cite{mingione2010} to obtain a proof of \Cref{th:main}. We always assume that $u$, $A$, and $G$ are solutions as in \Cref{th:main}. \\

We fix the parameter $\theta_0 \in (0,1)$ in \Cref{th:integrability_pharmonic} for the whole section. We will prove \Cref{th:main} for $\theta = \frac{\theta_0}{4}$, see \eqref{eq:def_theta}. The final result will follow from the fact that $\theta_0$ is arbitrary.

\subsection*{Step 1: Level-set decay}

\begin{lemma}\label{lm:levelset_decay}
	Let $\theta=\frac{\theta_0}{4}$. There exist a universal constant $\Gamma=\Gamma(n,N)>1$ such that the following holds. We denote $Q_\theta$ the inner cube of $B_\theta$.\\
	
	For every $T>1$, there exists $\eps_2=\eps_2(n,N,T)\in(0,1)$ such that the following holds. \\
	
	For any $\eps\in(0,\eps_2)$, there exists $\eta = \eta(n,N,\eps,T)\in (0,1)$ and $\gamma_2=\gamma_2(n,N,\eps,T)\in(0,1)$ such that, if
	\begin{align*}
		\|\g A\|_{L^n(B_1)} + \|\g A^{-1} \|_{L^n(B_1)} \leq \gamma_2,
	\end{align*}
	then the following holds. For any dyadic subcube $Q$ of $Q_\theta$ and any $\lambda>\lambda_0$, where 
	\begin{align}\label{def:lambda0}
		\lambda_0 := \frac{2^n}{|B_1| \theta_0^n} \|\g u\|_{L^{n-\ve}(B_1)}^{n-\eps},
	\end{align}
	the following holds.

	Assume that the predecessor $\tilde{Q}$ of $Q$ is contained in $Q_\theta$ and that the following inequality holds
	\begin{align}\label{eq:assumption_initial_cube}
		\left| Q\cap \left\{ x\in Q_{\theta} : M_{B_1}[|\g u|^{n-\ve}](x) > {\Gamma}T\lambda,\ M_{B_1}\left[ |G|^\frac{n-\ve}{n-1} \right](x) \leq \eta \lambda \right\} \right| > T^{-\frac{2n}{n-\eps}}|Q|.
	\end{align}
	Then $\tilde{Q}$ satisfies
	\begin{align}\label{eq:consequence_predecessor}
		\tilde{Q} \subset \left\{ x\in Q_{\theta}: M_{B_1}[|\g u|^{n-\ve}](x) > \lambda \right\}.
	\end{align}
\end{lemma}

\begin{proof}
	We will define $\theta$ later in \eqref{eq:def_theta}, for the moment, we consider only that $\theta\leq \frac{1}{4}$, in order to have $4Q\subset B_1$. By contradiction, we assume that \eqref{eq:assumption_initial_cube} is valid but \eqref{eq:consequence_predecessor} fails.\\
	Since \eqref{eq:consequence_predecessor} is wrong but $\tilde{Q}\subset B_\theta$, there exists $x_0\in \tilde{Q}$ such that
	\begin{align}\label{eq:maximal_gu_lambda}
		M_{B_1}[|\g u|^{n-\ve}](x_0) \leq \lambda.
	\end{align}
	Let $B$ the unique ball having $3Q$ as inner cube, $B \subset 4Q \subset B_1$. Then we have
	\begin{align*}
		\mvint_{B} |\g u|^{n-\eps} \leq \lambda.
	\end{align*}
	From \eqref{eq:assumption_initial_cube}, there exists $x_1 \in Q$ such that $M_{B_1}\left[ |G|^\frac{n-\eps}{n-1}\right](x_1) \leq \eta \lambda$. Since $x_1 \in B \subset B_1$, we also have
	\begin{align*}
		\mvint_B |G|^\frac{n-\eps}{n-1} \leq \eta \lambda.
	\end{align*}
	From \Cref{la:initial_estimate_difference}, there exists a radius $\rho\in(\frac{1}{2},\frac{3}{4})$ such that the $n$-harmonic extension $v\in W^{1,n}(\rho B;\R^N)$ of $u$ satisfies
	\begin{align}
		\mvint_{\rho B} |\g u - \g v|^{n-\ve}  &\leq C_0(\sigma,\eps)\mvint_{B} |G|^\frac{n-\eps}{n-1}
		+ \sigma \mvint_{B} |\g u|^{n-\ve}  \nonumber\\
		&\leq C_0(\sigma,\eps)\eta \lambda + \sigma\lambda.\label{eq:final_comparison}
	\end{align}
	Furthermore, it holds by \Cref{th:integrability_pharmonic}
	\begin{align}\label{eq:higher_int_gv1}
		\left(\mvint_{\frac{\theta_0}{2} B} |\g v|^{2n} \right)^\frac{1}{2n} \leq c(n) \left( \mvint_{\frac{1}{2} B} |\g v|^{n-\eps} \right)^\frac{1}{n-\eps}.
	\end{align}
	Combining \eqref{eq:maximal_gu_lambda} and \eqref{eq:final_comparison}, we deduce that
	\begin{align*}
		\left(\mvint_{\frac{1}{2} B} |\g v|^{n-\eps} \right)^\frac{1}{n-\eps} &\leq c\left(\mvint_{\rho B} |\g (v-u)|^{n-\eps} \right)^\frac{1}{n-\eps} + c\left(\mvint_B |\g u|^{n-\eps} \right)^\frac{1}{n-\eps}\\
		&\leq (\left( c(n) + C_0(\sigma,\eps)\eta \right) \lambda)^\frac{1}{n-\ve}.
	\end{align*}
	From \eqref{eq:higher_int_gv1}, we obtain
	\begin{align}\label{eq:higher_int_gv2}
		\mvint_{\frac{\theta_0}{2}B} |\g v|^{2n}  \leq c(n)((1+C_0(\sigma,\eps)\eta)\lambda)^\frac{2n}{n-\ve}.
	\end{align}
	Now, we have all the ingredients to estimate the quantity 
	\begin{align*}
		\left| \left\{ x\in Q : M_{B_1}[|\g u|^{n-\eps}](x) > {\Gamma}T\lambda \right\} \right|.
	\end{align*}
	First, we compare with the more restricted maximal function
	\begin{align*}
		{M_{B_{\theta_0/2}}}[|\g u|^{n-\eps}](x) := \sup \left\{ \mvint_{B(y,r)} |\g u|^{n-\eps} : x\in B(y,r)\subset B_{\theta_0/2}  \right\}.
	\end{align*}
	We obtain the following relation
	\begin{align}\label{eq:splitting_M}
		M_{B_1}[|\g u|^{n-\eps}](x) &= \max\left( {M_{B_{\theta_0/2}}}[|\g u|^{n-\eps}](x); \sup \left\{  \mvint_{B(y,r)} |\g u|^{n-\ve} : x\in B(y,r)\not\subset B_{\theta_0/2} \right\} \right).
	\end{align}
	Given $x\in B_{\theta_0/4}$, we estimate the second term using that a ball $B(y,r)\not\subset B_{\theta_0/2}$ containing $x$ must have a radius $r\geq \frac{\theta_0}{2}$: it holds
	\begin{align}\label{eq:second_term_splitting_M}
		\sup \left\{ \mvint_{B(x,r)} |\g u|^{n-\ve} : B(x,r)\not\subset B_{\theta_0/2} \right\} \leq \frac{2^n}{|B_1| \theta_0^n } \|\g u\|_{L^{n-\ve}(B_1)}^{n-\ve} =: \lambda_0.
	\end{align}
	We define 
	\begin{align}\label{eq:def_theta}
		\theta := \frac{\theta_0}{4}<\frac{1}{4}.
	\end{align}
	If $\lambda>\lambda_0$, we deduce from \eqref{eq:splitting_M} and \eqref{eq:second_term_splitting_M}:
	\begin{align*}
		\left| \left\{ x\in Q : M_{B_1}[|\g u|^{n-\eps}](x) > {\Gamma}T\lambda \right\} \right|  = & \left| \left\{ x\in Q : {M_{B_{\theta_0/2}}}[|\g u|^{n-\eps}](x) > {\Gamma}T\lambda \right\} \right| \\
		\leq & \left| \left\{ x\in Q : {M_{B_{\theta_0/2}}}[|\g v|^{n-\eps}](x) > \frac{1}{2^n} {\Gamma}T\lambda \right\} \right|\\
		&+ \left| \left\{ x\in Q : {M_{B_{\theta_0/2}}}[|\g (u-v)|^{n-\eps}](x) > \frac{1}{2^n}{\Gamma}T\lambda \right\} \right|.
	\end{align*}
	With standard estimates on maximal functions we find
	\begin{align*}
		\left| \left\{ x\in Q : M_{B_1}[|\g u|^{n-\eps}](x) > {\Gamma}T\lambda \right\} \right| \leq &\frac{C(n)}{({\Gamma}T\lambda)^\frac{2n}{n-\ve}} \int_Q |\g v|^{2n} + \frac{C(n)}{{\Gamma}T\lambda} \int_Q|\g(u-v)|^{n-\ve}.
	\end{align*}
	Combining this with \eqref{eq:final_comparison} and \eqref{eq:higher_int_gv2}, we arrive at
	\begin{align}
		\left| \left\{ x\in Q : M_{B_1}[|\g u|^{n-\eps}](x) > {\Gamma}T\lambda \right\} \right| \leq & \frac{C(n)|Q|}{({\Gamma}T\lambda)^\frac{2n}{n-\ve}} ((1+C_0(\sigma,\eps)\eta)\lambda)^\frac{2n}{n-\ve}+ \frac{C(n)|Q|}{{\Gamma}T\lambda}  (C_0(\sigma,\ve)\eta + c(n)\sigma)\lambda \nonumber \\
		\leq & \frac{C(n)|Q|}{(\Gamma T)^\frac{2n}{n-\ve}} (1+C_0(\sigma,\eps)\eta)^\frac{2n}{n-\ve}+ \frac{C(n)|Q|}{{\Gamma}T}  (C_0(\sigma,\ve)\eta + c(n)\sigma). \label{eq:level_set_Q1}
	\end{align}
	We first choose $\sigma=\sigma(n,T)$ small enough in order to obtain
	\begin{align*}
		c(n)\sigma \leq \frac{1}{2T^{\frac{2n}{n-1}-1} }\leq \frac{1}{2T^{\frac{2n}{n-\ve}-1} }.
	\end{align*}
	This choice fixes $\eps_2=\eps_1(n,T)$ thanks to \Cref{la:initial_estimate_difference}. Then, for any $\eps\in(0,\eps_2)$ we obtain a constant $\gamma_2 = \gamma_1(n,N,\eps,\sigma)$ and we choose $\eta=\eta(n,\eps,T)$ small enough, so that we obtain
	\begin{align*}
		C_0(\sigma,\ve)\eta + c(n)\sigma & \leq \frac{1}{T^{\frac{2n}{n-1}-1} } \leq  \frac{1}{T^{\frac{2n}{n-\ve}-1} },\\
		(1+C_0(\sigma,\eps)\eta)^\frac{2n}{n-\ve} & \leq 2.
	\end{align*}
	Coming back to \eqref{eq:level_set_Q1} with these choices, we obtain
	\begin{align}
		\left| \left\{ x\in Q : M_{B_1}[|\g u|^{n-\eps}](x) > {\Gamma}T\lambda \right\} \right| \leq \frac{C(n)|Q|}{T^\frac{2n}{n-\ve}} \left( \frac{1}{{\Gamma}^\frac{2n}{n-\eps}} + \frac{1}{{\Gamma}} \right) \leq \frac{C(n)|Q|}{T^\frac{2n}{n-\ve}} \left( \frac{1}{{\Gamma}^2} + \frac{1}{{\Gamma}} \right). \label{eq:level_set_Q2}
	\end{align}
	We now choose ${\Gamma}={\Gamma}(n)$ large enough to have
	\begin{align}\label{eq:def_A}
		C(n)\left( \frac{1}{{\Gamma}^2} + \frac{1}{\Gamma} \right) < \frac{1}{2}.
	\end{align}
	Hence, we obtain
	\begin{align*}
		\left| \left\{ x\in Q : M_{B_1}[|\g u|^{n-\eps}](x) > {\Gamma}T\lambda \right\} \right| \leq \frac{|Q|}{2 T^\frac{2n}{n-\eps}}.
	\end{align*}
	This is a contradiction to \eqref{eq:assumption_initial_cube}, and we can conclude.
\end{proof}

\subsection*{Step 2: Application of \Cref{la:size_subset}}
\begin{lemma}
	Let $\theta= \frac{\theta_0}{4}$. There exists a universal constant $\Gamma = \Gamma(n,N)>1$ such that the following holds. Let $Q_\theta$ be the inner cube of $B_\theta$.\\
	
	For every $T>1$, there exists $\eps_2=\eps_2(n,N,T)\in(0,1)$ such that the following holds. \\
	
	For every $\eps\in(0,\eps_2)$, we define 
	\begin{align*}
		\lambda_1 := \max\left( \frac{2^n \Gamma}{|B_1|\theta_0^n}, \frac{1}{ |B_\theta|} \right) T^\frac{2n}{n-\eps} \|\g u\|_{L^{n-\eps}(B_1)}^{n-\eps}.
	\end{align*}
	There exists $\eta = \eta(n,N,\eps,T)\in(0,1)$  and $\gamma_2=\gamma_2(n,N,\eps,T)\in(0,1)$ such that, if
	\begin{align*}
		\|\g A\|_{L^n(B_1)} + \|\g A^{-1} \|_{L^n(B_1)} \leq \gamma_2,
	\end{align*}
	then for every $k\in\N$, it holds
	
	\begin{align}
		\left| \left\{ x\in Q_\theta : M_{B_1}[|\g u|^{n-\ve}](x) > ({\Gamma}T)^{k+1} \lambda_1 \right\} \right| \leq & \frac{1}{T^{\frac{2n}{n-\eps}} } \left| \left\{ x\in Q_\theta: M_{B_1}[|\g u|^{n-\ve}](x) > ({\Gamma}T)^k \lambda_1 \right\} \right| \nonumber \\
		& + \left| \left\{ x\in Q_\theta : M \left[ |G|^\frac{n-\ve}{n-1} \right](x) >\eta ({\Gamma}T)^k \lambda_1 \right\} \right|. \label{eq:level_set_growth}
	\end{align}
	
\end{lemma}

\begin{proof}
	We consider the cases where $\lambda = ({\Gamma}T)^k\lambda_1$ in \Cref{lm:levelset_decay} for any $k\in\N^*$ and for some $\lambda_1$ to be chosen later. The goal is to apply \Cref{la:size_subset} with $Q_0 = Q_\theta$, $\delta= T^{-\frac{2n}{n-\eps}}$ and the sets 
	\begin{align}
		X & = \left| \left\{ x\in Q_\theta : M_{B_1}[|\g u|^{n-\ve}](x) > ({\Gamma}T)^{k+1} \lambda_1, M_{B_1}\left[ |G|^\frac{n-\ve}{n-1} \right](x) \leq \eta ({\Gamma}T)^k \lambda_1 \right\} \right|, \label{eq:def_X}\\
		Y & = \left| \left\{ x\in Q_\theta: M_{B_1}[|\g u|^{n-\ve}](x) > ({\Gamma}T)^k \lambda_1 \right\} \right|. \label{eq:def_Y}
	\end{align}
	
	From \Cref{lm:levelset_decay}, \Cref{It2} in \Cref{la:size_subset} is satisfied from $\lambda_1=\lambda_0$. However, we need to increase $\lambda_1$ in order to obtain \Cref{It1}.\\
	
	To do so, we first consider $\lambda_1$ of the form $\alpha T^\frac{2n}{n-\eps} \lambda_0$, for some universal constant $\alpha=\alpha(n,N,\theta_0)\geq \Gamma$, in order to satisfy \Cref{It1} in \Cref{la:size_subset}. From the definition of $\lambda_0$ in \eqref{def:lambda0}, we have from standard estimates of maximal functions
	\begin{align*}
		\left| \left\{ x\in Q_\theta : M_{B_1}[|\g u|^{n-\eps}](x) > \alpha T^\frac{2n}{n-\eps} \lambda_0 \right\} \right| & \leq \frac{1}{\alpha T^\frac{2n}{n-\eps} \lambda_0} \int_{B_1} |\g u|^{n-\eps} \\
		&\leq \frac{|B_1|\theta_0^n}{2^n \alpha T^\frac{2n}{n-\eps}}.
	\end{align*}
	We define $\alpha=\alpha(n,N,\theta_0)$ by the relation
	\begin{align*}
		\alpha = \max\left( {\Gamma}, \frac{|B_1|\theta_0^n}{2^n |Q_\theta|} \right).
	\end{align*}
	That is, we have
	\begin{align*}
		\left| \left\{ x\in Q_\theta : M_{B_1}[|\g u|^{n-\eps}](x) > \alpha T^\frac{2n}{n-\eps} \lambda_0 \right\} \right| \leq \frac{|Q_\theta|}{T^\frac{2n}{n-\eps}}.
	\end{align*}
	We define $\lambda_1\geq \lambda_0$ by the relation
	\begin{align}\label{eq:def_lambda1}
		\lambda_1 := \alpha T^\frac{2n}{n-\eps} \lambda_0 = \max\left( \frac{1}{|Q_\theta|},\frac{2^n {\Gamma}}{|B_1|\theta_0^n} \right)T^\frac{2n}{n-\eps}\int_{B_1} |\g u|^{n-\eps}.
	\end{align}
	We now check \Cref{It1} of \Cref{la:size_subset} for the set $X$ defined in \eqref{eq:def_X} and $\lambda_1$ defined in \eqref{eq:def_lambda1}. Fix an integer $k\geq 1$. Since $\Gamma T \geq 1$, we have
	\begin{align*}
		\left| \left\{ x\in Q_\theta : M_{B_1}[|\g u|^{n-\eps}](x) > ({\Gamma}T)^k\lambda_1 \right\} \right| \leq & 	\left| \left\{ x\in Q_\theta : M_{B_1}[|\g u|^{n-\eps}](x) > \lambda_1 \right\} \right|\\
		\leq & \frac{|Q_\theta|}{T^\frac{2n}{n-\eps}}.
	\end{align*}
	
	Thus, we can apply \Cref{la:size_subset} with the choices $Q_0=Q_\theta$, $\delta = T^{-\frac{2n}{n-\eps}}$ and $X,Y$ defined in \eqref{eq:def_X}-\eqref{eq:def_Y}. We obtain
	\begin{align*}
		& \left| \left\{ x\in Q_\theta : M_{B_1}[|\g u|^{n-\ve}](x) > ({\Gamma}T)^{k+1} \lambda_1, M_{B_1}\left[ |G|^\frac{n-\ve}{n-1} \right](x) \leq \eta ({\Gamma}T)^k \lambda_1 \right\} \right| \\
		\leq & \frac{1}{T^{\frac{2n}{n-\eps}} } \left| \left\{ x\in Q_\theta: M_{B_1}[|\g u|^{n-\ve}](x) > ({\Gamma}T)^k \lambda_1 \right\} \right|.
	\end{align*}
	Therefore, it holds
	\begin{align*}
		& \left| \left\{ x\in Q_\theta : M_{B_1}[|\g u|^{n-\ve}](x) > ({\Gamma}T)^{k+1} \lambda_1 \right\} \right|  \\
		\leq & \left| \left\{ x\in Q_\theta : M_{B_1}[|\g u|^{n-\ve}](x) > ({\Gamma}T)^{k+1} \lambda_1, M_{B_1}\left[ |G|^\frac{n-\ve}{n-1} \right](x) \leq \eta ({\Gamma}T)^k \lambda_1 \right\} \right|  \\
		& +  \left| \left\{ x\in Q_\theta : M \left[ |G|^\frac{n-\ve}{n-1} \right] (x) >\eta ({\Gamma}T)^k \lambda_1 \right\} \right|  \\
		\leq & \frac{1}{T^{\frac{2n}{n-\eps}} } \left| \left\{ x\in Q_\theta: M_{B_1}[|\g u|^{n-\ve}](x) > ({\Gamma}T)^k \lambda_1 \right\} \right|  \\
		& + \left| \left\{ x\in Q_\theta : M \left[ |G|^\frac{n-\ve}{n-1} \right](x) >\eta ({\Gamma}T)^k \lambda_1 \right\} \right|. 
	\end{align*}
\end{proof}

\subsection*{Step 3 : Lorentz spaces estimates}
\begin{proof}[Proof of \Cref{th:main}]
	Following the notations of \cite{mingione2010}, we define for any $H\geq 0$
	\begin{align*}
		\mu_1(H) &:= \left| \left\{ x\in Q_\theta : M_{B_1}[|\g u|^{n-\ve}](x) > H \right\} \right|,\\
		\mu_2(H) &:= \left| \left\{ x\in Q_\theta : M_{B_1}\left[ |G|^\frac{n-\ve}{n-1} \right](x) >H \right\} \right|.
	\end{align*}
	For any integer $k\geq 1$, we write the estimate \eqref{eq:level_set_growth} in terms of $\mu_1$ and $\mu_2$
	\begin{align}\label{eq:pointwise_mu}
		\mu_1(({\Gamma}T)^{k+1} \lambda_1) \leq \mu_2(\eta ({\Gamma}T)^k \lambda_1) + T^{-\frac{2n}{n-\eps}}\mu_1(({\Gamma}T)^k \lambda_1).
	\end{align}
	We define the sequences 
	\begin{align*}
		a_k &:= \mu_1(({\Gamma}T)^k \lambda_1),\\
		b_k & := \mu_2( ({\Gamma}T)^k \eta\lambda_1).
	\end{align*}
	If $G\in L^{(\frac{n}{n-1},q)}$, then $|G|^\frac{n-\eps}{n-1} \in L^{(\frac{n}{n-\eps},q\frac{n-1}{n-\eps})}$. By \Cref{th:boundedness_maximal_function}, it holds $M_{B_1}[|G|^\frac{n-\eps}{n-1}] \in L^{(\frac{n}{n-\eps},q\frac{n-1}{n-\eps})}$.
	We write this in terms of $\mu_2$
	\begin{align*}
		\left\|M_{B_1}[|G|^\frac{n-\eps}{n-1}] \right\|_{L^{\left( \frac{n}{n-\eps},q\frac{n-1}{n-\eps} \right)}(Q_\theta) }^{q\frac{n-1}{n-\eps}} = q\frac{n-1}{n-\eps}\int_0^\infty \left( H^\frac{n}{n-\eps} \mu_2(H) \right)^{q\frac{n-1}{n}} \frac{dH}{H} <\infty.
	\end{align*}
	We derive an estimate on the sequence $b_k$ form the above inequality. We start with the following decomposition of the integral
	\begin{align*}
		\int_{\eta\lambda_1}^\infty \left( H^\frac{n}{n-\eps} \mu_2(H) \right)^{q\frac{n-1}{n}} \frac{dH}{H} = &\sum_{k=0}^\infty \int_{({\Gamma}T)^k \eta\lambda_1}^{({\Gamma}T)^{k+1} \eta\lambda_1 }\left( H^\frac{n}{n-\eps} \mu_2(H) \right)^{q\frac{n-1}{n}} \frac{dH}{H}.
	\end{align*}
	Since $\mu_2$ is a nonincreasing function, we obtain
	\begin{align*}
		\int_{\eta\lambda_1}^\infty \left( H^\frac{n}{n-\eps} \mu_2(H) \right)^{q\frac{n-1}{n}} \frac{dH}{H} \geq & \sum_{k=0}^\infty b_{k+1}^{q\frac{n-1}{n}} \int_{({\Gamma}T)^k \eta\lambda_1}^{({\Gamma}T)^{k+1} \eta\lambda_1 }\left( H^\frac{n}{n-\eps}  \right)^{q\frac{n-1}{n}} \frac{dH}{H} \\
		\geq & \sum_{k=0}^\infty b_{k+1}^{q\frac{n-1}{n}} \int_{({\Gamma}T)^k \eta\lambda_1}^{({\Gamma}T)^{k+1} \eta\lambda_1 } H^{q\frac{n-1}{n-\eps} - 1} dH.
	\end{align*}
	If $q<\frac{n}{n-1}$, then up to reducing $\eps$, we also have $q<\frac{n-\eps}{n-1}$. Therefore, $q\frac{n-1}{n-\eps} - 1 <0$ and we obtain the desired estimate on the summability of $b_k$
	\begin{align*}
		\int_{\eta\lambda_1}^\infty \left( H^\frac{n}{n-\eps} \mu_2(H) \right)^{q\frac{n-1}{n}} \frac{dH}{H} \geq & \sum_{k=0}^\infty b_{k+1}^{q\frac{n-1}{n}} \left( ({\Gamma}T)^{k+1} \eta\lambda_1 \right)^{q\frac{n-1}{n-\eps} - 1}  \int_{({\Gamma}T)^k \eta\lambda_1 }^{({\Gamma}T)^{k+1} \eta\lambda_1 } dH \\
		\geq & ({\Gamma}T-1) \sum_{k=0}^\infty b_{k+1}^{q\frac{n-1}{n}} \left( ({\Gamma}T)^{k+1} \eta\lambda_1 \right)^{q\frac{n-1}{n-\eps}} .
	\end{align*}
	Since $T>1$ and ${\Gamma}>1$, we have the following inequality
	\begin{align}\label{eq:summability_bk}
		\sum_{k=0}^\infty b_{k+1}^{q\frac{n-1}{n}} \left( ({\Gamma}T)^{k+1} \eta\lambda_1 \right)^{q\frac{n-1}{n-\eps}} \leq c(n)\left\|M_{B_1}[|G|^\frac{n-\eps}{n-1}] \right\|_{L^{\left( \frac{n}{n-\eps},q\frac{n-1}{n-\eps} \right)}(Q_\theta) }^{q\frac{n-1}{n-\eps}}.
	\end{align}
	We now write \eqref{eq:pointwise_mu} in terms of the sequences $(a_k)_k$ and $(b_k)_k$. We obtain for any $k\geq 0$
	\begin{align*}
		a_{k+2} \leq b_{k+1} + \frac{ a_{k+1} }{ T^{\frac{2n}{n-\eps}} } \leq b_{k+1} + \frac{ a_{k+1} }{ T^2 }. 
	\end{align*}
	We raise to the power $q\frac{n-1}{n}$ and multiply by $\left( ({\Gamma}T)^{k+2} \eta\lambda_1 \right)^{q\frac{n-1}{n-\eps}}$ to obtain
	\begin{align}
		\left( ({\Gamma}T)^{k+2} \eta\lambda_1 \right)^{q\frac{n-1}{n-\eps}} a_{k+2}^{q\frac{n-1}{n}} \leq & \left( ({\Gamma}T)^{k+2} \eta\lambda_1 \right)^{q\frac{n-1}{n-\eps}} 2^{q\frac{n-1}{n}} \left( b_{k+1}^{q\frac{n-1}{n}} + T^{-2 q\frac{n-1}{n}} a_{k+1}^{q\frac{n-1}{n}} \right) \nonumber\\
		\leq & 2^{q\frac{n-1}{n}} ({\Gamma}T)^{q\frac{n-1}{n-\eps}} \left( b_{k+1}^{q\frac{n-1}{n}} \left( ({\Gamma}T)^{k+1} \eta\lambda_1 \right)^{q\frac{n-1}{n-\eps}} \right)  \nonumber\\
		& + 2^{q\frac{n-1}{n}} ({\Gamma}T)^{q\frac{n-1}{n-\eps}} T^{-2 q\frac{n-1}{n}} \left( a_{k+1}^{q\frac{n-1}{n}} \left( ({\Gamma}T)^{k+1} \eta\lambda_1 \right)^{q\frac{n-1}{n-\eps}} \right). \label{eq:pointwise_levelset}
	\end{align}
	Up to reducing $\eps$ again, it holds:
	\begin{align*}
		q\frac{n-1}{n-\eps}-2 q\frac{n-1}{n} &= q(n-1) \left( \frac{1}{n-\eps} - \frac{2}{n} \right) \\
		&\leq q(n-1) \left( \frac{3}{2n} - \frac{2}{n} \right)\\
		&\leq q(n-1)\frac{1}{2n}<0.
	\end{align*}
	Therefore, using that ${\Gamma}={\Gamma}(n)$ (see \eqref{eq:def_A}), we can choose $T=T(n,q)>1$ such that
	\begin{align}\label{eq:def_T}
		2^{q\frac{n-1}{n}} {\Gamma}^{q\frac{n-1}{n-\eps}} T^{q\frac{n-1}{n-\eps}-2 q\frac{n-1}{n}} \leq 2^{q\frac{n-1}{n}} {\Gamma}^q T^{q(n-1) \frac{-1}{2n}} \leq \frac{1}{4}.
	\end{align}
	Coming back to \eqref{eq:pointwise_levelset}, we obtain
	\begin{align*}
		\left( ({\Gamma}T)^{k+2} \eta\lambda_1 \right)^{q\frac{n-1}{n-\eps}} a_{k+2}^{q\frac{n-1}{n}} \leq & c(n,q)\left(  b_{k+1}^{q\frac{n-1}{n}} \left( ({\Gamma}T)^{k+1} \eta\lambda_1 \right)^{q\frac{n-1}{n-\eps}} \right)\\
		& + \frac{1}{4} \left( a_{k+1}^{q\frac{n-1}{n}} \left( ({\Gamma}T)^{k+1} \eta\lambda_1 \right)^{q\frac{n-1}{n-\eps}} \right).
	\end{align*}
	Given any integer $K\geq 1$, we sum the above inequalities for $k$ between $0$ and $K$ to obtain
	\begin{align*}
		\sum_{k=0}^K \left( ({\Gamma}T)^{k+2} \eta\lambda_1 \right)^{q\frac{n-1}{n-\eps}} a_{k+2}^{q\frac{n-1}{n}} \leq & c(n,q)\sum_{k=0}^K \left(  b_{k+1}^{q\frac{n-1}{n}} \left( ({\Gamma}T)^{k+1} \eta\lambda_1 \right)^{q\frac{n-1}{n-\eps}} \right)\\
		& + \frac{1}{4} \sum_{k=0}^K \left( a_{k+1}^{q\frac{n-1}{n}} \left( ({\Gamma}T)^{k+1} \eta\lambda_1 \right)^{q\frac{n-1}{n-\eps}} \right).
	\end{align*}
	By reindexing the sums over the $a_k$, it holds
	\begin{align*}
		\sum_{k=2}^{K+2} \left( ({\Gamma}T)^k \eta\lambda_1 \right)^{q\frac{n-1}{n-\eps}} a_k^{q\frac{n-1}{n}} \leq & c(n,q)\sum_{k=0}^K \left(  b_{k+1}^{q\frac{n-1}{n}} \left( ({\Gamma}T)^{k+1} \eta\lambda_1 \right)^{q\frac{n-1}{n-\eps}} \right)\\
		& + \frac{1}{4} \sum_{k=1}^{K+1} \left( a_k^{q\frac{n-1}{n}} \left( ({\Gamma}T)^k \eta\lambda_1 \right)^{q\frac{n-1}{n-\eps}} \right).
	\end{align*}
	Therefore, the last sum of the right-hand side can be reabsorbed by the left-hand side, only the term in $k=1$ remains on the right-hand side
	\begin{align*}
		\frac{3}{4}\sum_{k=2}^{K+2} \left( ({\Gamma}T)^k \eta\lambda_1 \right)^{q\frac{n-1}{n-\eps}} a_k^{q\frac{n-1}{n}}\leq & c(n,q)\sum_{k=0}^K \left(  b_{k+1}^{q\frac{n-1}{n}} \left( ({\Gamma}T)^{k+1} \eta\lambda_1 \right)^{q\frac{n-1}{n-\eps}} \right)\\
		& + \frac{1}{4} \left( a_1^{q\frac{n-1}{n}} \left( {\Gamma}T \eta\lambda_1 \right)^{q\frac{n-1}{n-\eps}} \right).
	\end{align*}
	We now consider the limit $K\to \infty$
	\begin{align*}
		\frac{3}{4}\sum_{k=2}^{\infty} \left( ({\Gamma}T)^k \eta\lambda_1 \right)^{q\frac{n-1}{n-\eps}} a_k^{q\frac{n-1}{n}} \leq & c(n,q)\sum_{k=0}^\infty \left(  b_{k+1}^{q\frac{n-1}{n}} \left( ({\Gamma}T)^{k+1} \eta\lambda_1 \right)^{q\frac{n-1}{n-\eps}} \right)\\
		& + \frac{1}{4} \left( a_1^{q\frac{n-1}{n}} \left( {\Gamma}T \eta\lambda_1 \right)^{q\frac{n-1}{n-\eps}} \right).
	\end{align*}
	Using $a_1\leq |B_1|$, the expression of $\lambda_1$ in \eqref{eq:def_lambda1} and the summability of $(b_k)_k$ in \eqref{eq:summability_bk}, we deduce that
	\begin{align*}
		\frac{3}{4}\sum_{k=2}^{\infty} \left( ({\Gamma}T)^k \eta\lambda_1 \right)^{q\frac{n-1}{n-\eps}} a_k^{q\frac{n-1}{n}} \leq & c(n,q,\theta_0) \left( \left\| M_{B_1}[ |G|^\frac{n-\ve}{n-1} ] \right\|_{L^{(\frac{n}{n-\ve}, q \frac{n-1}{n-\ve} )}(B_1) }^{q\frac{n-1}{n-\eps}} + \|\g u\|_{L^{n-\eps}(B_1)}^{q(n-1)} \right).
	\end{align*}
	From \Cref{th:boundedness_maximal_function}, we deduce that
	\begin{align*}
		\sum_{k=2}^{\infty} \left( ({\Gamma}T)^k \eta\lambda_1 \right)^{q\frac{n-1}{n-\eps}} a_k^{q\frac{n-1}{n}} \leq & c(n,q,\theta_0) \left( \left\| G \right\|_{L^{(\frac{n}{n-1}, q  )}(B_1) }^q + \|\g u\|_{L^{n-\eps}(B_1)}^{q(n-1)} \right) .
	\end{align*}
	Arguing similar to \eqref{eq:summability_bk}, we obtain from the above estimate an estimate in Lorentz spaces of $M_{B_1}[|\g u|^{n-\eps}]$
	\begin{align*}
		&\left\| M_{B_1}[|\g u|^{n-\eps}] \right\|_{L^{(\frac{n}{n-\ve}, q \frac{n-1}{n-\ve} )}(Q_\theta) }^{q\frac{n-1}{n-\eps}} \\
		= & q \frac{n-1}{n-\ve}\int_0^\infty \left( H^\frac{n}{n-\eps} \mu_1(H) \right)^{q\frac{n-1}{n}} \frac{dH}{H} \\
		=& q \frac{n-1}{n-\ve} \left( \int_0^{({\Gamma}T)^2 \lambda_1 } \left( H^\frac{n}{n-\eps} \mu_1(H) \right)^{q\frac{n-1}{n}} \frac{dH}{H} + \sum_{k=2}^\infty \int_{({\Gamma}T)^k \lambda_1 }^{({\Gamma}T)^{k+1} \lambda_1} \left( H^\frac{n}{n-\eps} \mu_1(H) \right)^{q\frac{n-1}{n}} \frac{dH}{H} \right)\\
		\leq & q \left( |B_1|^{q\frac{n-1}{n}} \left( ({\Gamma}T)^2 \lambda_1 \right)^{q\frac{n-1}{n-\eps}} + ({\Gamma}T-1)\sum_{k=2}^\infty a_k^{q\frac{n-1}{n}} \left( ({\Gamma}T)^k \lambda_1 \right)^{q\frac{n-1}{n-\eps}} \right)\\
		\leq & c(n,q,\theta_0)\left(\left\| G \right\|_{L^{(\frac{n}{n-1}, q  )}(B_1) }^q + \|\g u\|_{L^{n-\eps}(B_1)}^{q(n-1)} \right).
	\end{align*}
	We conclude the proof of \Cref{th:main}
	\begin{align*}
		\|\g u\|_{L^{(n,q(n-1))}(Q_\theta) }^{q(n-1)} \leq c(n,q,\theta_0)\left(\left\| G \right\|_{L^{(\frac{n}{n-1}, q  )}(B_1) }^q + \|\g u\|_{L^{n-\eps}(B_1)}^{q(n-1)} \right).
	\end{align*}
\end{proof}

\section{Optimality of \Cref{co:nLap_H1}}\label{s:optimal}

In this section, we study the regularity of the examples obtained in \cite{F95}. Firoozye proved that for any $\alpha\in\left( 0, \frac{n-2}{n-1}\right)$, the function $u_\alpha(x) = \log(1/|x|)^{\alpha}$ is a solution to $\lap_n u \in \HI^1_{loc}$ on a ball $B_{1/2}\subset \R^n$.

\begin{lemma}
	For every $\frac{1}{1-\alpha}<q$, it holds $\g u_\alpha\in L^{(n,q)}$.
\end{lemma}

\begin{remark}
	Since $\alpha<\frac{n-2}{n-1}$, it holds $\frac{1}{1-\alpha} < n-1$. In particuler, it holds $\g u_\alpha\in L^{(n,n-1)}(B_{1/2})$ for any $\alpha$. Furthermore, we have 
	\begin{align*}
		\frac{1}{1-\alpha} \xrightarrow[\alpha\to \frac{n-2}{n-1}]{}{n-1}.
	\end{align*}
	Thus, $L^{(n,n-1)}$ is the maximal integrability which is common to every $\g u_\alpha$.
\end{remark}

\begin{proof}
	The norm of the gradient of $u_\alpha$ is given by
	\begin{align*}
		\forall x\in B_{1/2},\ \ \ |\g u_\alpha(x)| = \frac{\alpha}{|x|} \log\left( \frac{1}{|x|}\right)^{\alpha-1}.
	\end{align*}
	The nondecreasing rearrangement of $|\g u_\alpha|$ is given by
	\begin{align*}
		\forall t\in(0,|B_{1/2}|),\ \ \ f(t) = \frac{\alpha}{|B_1|^{-\frac{1}{n}} t^\frac{1}{n}} \log\left( \frac{1}{|B_1|^{-\frac{1}{n}} t^\frac{1}{n}} \right)^{\alpha-1}.
	\end{align*}
	The map $\g u_\alpha \in L^{(n,q)}$ if and only if 
	\begin{align*}
		\int_0^{|B_{1/2}|} \left( t^\frac{1}{n} f(t) \right)^q \frac{dt}{t} <\infty.
	\end{align*}
	This is equivalent to 
	\begin{align*}
		\int_0^{1/2} \frac{dt}{t |\log(t)|^{q(1-\alpha)}} < \infty.
	\end{align*}
	This is true if and only if $q(1-\alpha)>1$. Hence, it holds $\g u_\alpha \in L^{(n,q)}(B_{1/2})$ for every $q>\frac{1}{1-\alpha}$. 
\end{proof}

\bibliographystyle{abbrv}
\bibliography{biblio}

\end{document}